\documentclass[11pt]{article} 
\usepackage{amsfonts,amsmath,latexsym,amssymb,mathrsfs,amsthm,comment}
\usepackage{graphicx}
\usepackage{xcolor}
\usepackage{booktabs}

\numberwithin{equation}{section}

\let\OLDthebibliography\thebibliography
\renewcommand\thebibliography[1]{
  \OLDthebibliography{#1}
  \setlength{\parskip}{1pt}
  \setlength{\itemsep}{0pt plus 0.0ex}
}

\def\numberlikeadb{\global\def\theequation{\thesection.\arabic{equation}}}
\numberlikeadb
\newtheorem{theorem}{Theorem}[section]

\newtheorem{proposition}[theorem]{Proposition}
\newtheorem{remark}[theorem]{Remark}

\usepackage[multiple]{footmisc}

\usepackage{lscape}
\usepackage{caption}
\usepackage{multirow}
\allowdisplaybreaks
\begin{document}

\title{On P\'olya's 4D random walk constant}

\author{Robert E. Gaunt\footnote{Department of Mathematics, The University of Manchester, Oxford Road, Manchester M13 9PL, UK, robert.gaunt@manchester.ac.uk, ORCID 0000-0001-6187-0657}}

\date{} 
\maketitle

\vspace{-5mm}

\begin{abstract} The return probability in the simple symmetric random walk on the 4-dimensional lattice $\mathbb{Z}^4$ is evaluated in closed form in terms of the generalized hypergeometric function.

\end{abstract}



\section{Introduction}

A celebrated theorem of P\'olya \cite{p21} states that the simple symmetric random walk on $\mathbb{Z}^d$ is recurrent in dimensions 1 and 2 but transient in dimension $d\geq3$. A fundamental problem is to calculate the return probability, which we denote by $p(d)$, for $d\geq3$; see \cite[Section 5.9]{finch} for an overview of some of the contributions to this problem. Exact formulas for $p(3)$, the return probability in dimension 3, were given in \cite{d54,mw40,watson}, and this line of research culminated in the simple closed-form formula 
\[p(3)=1-1/u(3)=0.3405373295\ldots,\]
where
\begin{align}
u(3)&=\frac{\sqrt{6}}{32\pi^3}\Gamma\Big(\frac{1}{24}\Big)\Gamma\Big(\frac{5}{24}\Big)\Gamma\Big(\frac{7}{24}\Big)\Gamma\Big(\frac{11}{24}\Big) \label{for1} \\
&=\frac{\sqrt{3}-1}{32\pi^3}\bigg[\Gamma\Big(\frac{1}{24}\Big)\Gamma\Big(\frac{11}{24}\Big)\bigg]^2\label{for2}\\
&=1.5163860591\ldots. \nonumber
\end{align}
Formula (\ref{for1}) was derived in \cite{gz77}, but the value printed there included an erroneous additional factor of $384\pi$, as pointed out in \cite{z11}. The simplified expression (\ref{for2}) was obtained by \cite{jz05} through an application of a gamma function identity given in \cite{bz92} to formula (\ref{for1}). The numerical values presented above for $p(3)$ and $u(3)$ are given in \cite[p.\ 322]{finch}, whilst numerical evaluations for $p(d)$ for $4\leq d\leq 8$ are tabulated in \cite[p.\ 323]{finch}.

More generally, in dimension $d\geq3$, \cite{gnp24} derived a closed-form expression for the return probability $p(d)$ in terms of the Lauricella function of type $C$: for $d\geq3$,
   \begin{equation*} 
	    p(d)=1-1/u(d),
	 \end{equation*}
where     
   \begin{equation} \label{FX0}
      u(d) = F_C^{(d)} \left( 1, \frac12; 1, \ldots, 1; \frac {1}{d^2}, \ldots, \frac {1}{d^2} \right),
   \end{equation}
and the Lauricella function $F_C$ (see \cite{e78,l93}) is defined by
\begin{align*}
\displaystyle
F_C^{(d)} \left( a, b; c_1, \ldots, c_d; x_1, \ldots, x_d \right)  =
\sum_{k_1 \geq 0} \cdots \sum_{k_d \geq 0}
\frac {(a)_{k_1 + \cdots + k_d} (b)_{k_1 + \cdots + k_d}}
{\left( c_1 \right)_{k_1} \cdots \left( c_d \right)_{k_d}}
\frac {x_1^{k_1} \cdots x_d^{k_d}}
{k_1! \cdots k_d!},
\end{align*}
with $(v)_0=1$ and $(v)_k = v (v + 1) \cdots (v + k - 1)$, $k\geq1$, denoting the Pochhammer symbol. 

Now that a closed-form formula for the return probability for general $d\geq3$ has been achieved in terms of a known special function, a natural further question is whether one can obtain closed-form expressions in terms of a simpler special function in a dimension greater than 3. In this note, we address this problem by obtaining closed-form expressions for the return probability in dimension 4 in terms of the ${}_{q+1}F_q$ generalized hypergeometric function, defined via the power series
\begin{equation}
\label{seriesrep}
{}_{q+1}F_q\bigg({a_1,\ldots,a_{q+1} \atop b_1,\ldots,b_q} \; \bigg| \;x\bigg)=\sum_{k=0}^\infty\frac{(a_1)_k\cdots(a_{q+1})_k}{(b_1)_k\cdots(b_q)_k}\frac{x^k}{k!},
\end{equation}
which is absolutely convergent for $|x|<1$, and is also absolutely convergent for $|x|=1$ if $\mathrm{Re}(\sum_{i=1}^qb_i-\sum_{j=1}^{q+1}a_j)>0$ (see \cite[Section 16.2]{olver}).  Our formulas for the return probability offer a considerable simplification of the formula for $p(4)$ obtained from simply setting $d=4$ in the general formula (\ref{FX0}), since the generalized hypergeometric function is defined through a single infinite series and is a more elementary and widely used special function than the Lauricella function of type $C$.  

\section{Results and proofs}

The following theorem is the main result of this note.

\begin{theorem}\label{thm1.1} In dimension 4, the return probability is given by
\begin{equation*}
p(4)=1-1/u(4)=0.1932016732\ldots, 
\end{equation*}
where
\begin{align}
u(4)&=\frac{3\Gamma(\frac{1}{3})^9}{16\pi^6}\,{}_4F_3\bigg({\frac{1}{6},\frac{1}{3},\frac{1}{3},\frac{1}{2} \atop \frac{2}{3},\frac{5}{6},\frac{5}{6}} \; \bigg| \;1\bigg)-\frac{64\sqrt{3}\,\pi^3}{3\Gamma(\frac{1}{3})^9}\,{}_4F_3\bigg({\frac{1}{2},\frac{2}{3},\frac{2}{3},\frac{5}{6} \atop \frac{7}{6},\frac{7}{6},\frac{4}{3}} \; \bigg| \;1\bigg) \label{4f3} \\
&=\frac{\Gamma(\frac{1}{3})^6}{2^{5/3}\pi^4}\,{}_5F_4\bigg({\frac{1}{3},\frac{1}{3},\frac{1}{2},\frac{1}{2},\frac{1}{2} \atop \frac{5}{6},\frac{5}{6},\frac{5}{6},1} \; \bigg| \;1\bigg) \label{5f4}\\
&=1.2394671218\ldots. \nonumber
\end{align}
\end{theorem} 

\begin{proof} We shall prove the theorem by combining several integral identities from the existing literature. By \cite[Eq.\ (2.11a)]{m56}, the return probability is given by $p(4)=1-1/u(4)$, where
\begin{align}
u(4)&=\frac{1}{4\pi^4}\int_{(-\pi,\pi)^4} \frac{\mathrm{d}x_1\, \mathrm{d}x_2\, \mathrm{d}x_3 \, \mathrm{d}x_4}{4-\sum_{k=1}^4\cos x_k}. \label{wat2}
\end{align} 
 From equations (2.1) and (2.3) of \cite{z19} we have that
\begin{equation}\label{fff}
\frac{1}{4\pi^4}\int_{(-\pi,\pi)^4} \frac{\mathrm{d}x_1\, \mathrm{d}x_2\, \mathrm{d}x_3 \, \mathrm{d}x_4}{4-\sum_{k=1}^4\cos x_k}= \frac{16}{\pi^2}\int_0^\infty tI_0^3(t) K_0(t)^3\,\mathrm{d}t, 
\end{equation}
whilst we also have the following integral identity
\begin{equation}\label{ggg}\int_0^\infty tI_0^3(t) K_0(t)^3\,\mathrm{d}t=\frac{3}{\pi^2}\int_0^\infty tI_0(t) K_0(t)^5\,\mathrm{d}t 
\end{equation}
(see \cite[p.\ 471]{z19}). Here, $I_0 (\cdot)$ and $K_0(\cdot)$ are the zeroth order modified Bessel functions of first and second kinds, respectively (see \cite[Chapter 10]{olver}). Combining equations (\ref{wat2}), (\ref{fff}) and (\ref{ggg}) now gives that
\begin{equation}\label{hhh}
u(4)=\frac{48}{\pi^4}\int_0^\infty tI_0(t) K_0(t)^5\,\mathrm{d}t.  
\end{equation}
Applying the formula of \cite[Proposition 2.3]{z19} for the integral $\int_0^\infty tI_0(t) K_0(t)^5\,\mathrm{d}t$  to (\ref{hhh}) now yields formula (\ref{4f3}).

We now derive formula (\ref{5f4}). From Propositions 2.3 and 2.4 of \cite{z19}, we can infer that
\begin{align*}
u(4)=\frac{48}{\pi^4}\int_0^\infty tI_0(t) K_0(t)^5\,\mathrm{d}t=\frac{4\sqrt{3}}{9\pi}\int_0^1\frac{1}{\sqrt{1-x}}\bigg[{}_2F_1\bigg({\frac{1}{3},\frac{2}{3} \atop 1} \; \bigg| \;x\bigg)\bigg]^2 \,\mathrm{d}x,  
\end{align*}
and evaluating the latter integral using the integral formula (2.28) of Proposition 3.1 of \cite{z19} yields the expression (\ref{5f4}). This completes the proof.
\end{proof} 

\begin{remark}
Expression (\ref{5f4}) can alternatively be deduced from formula (\ref{4f3}) via an application of equation (2.3) of \cite{m12}. The direct application of equation (2.3) of \cite{m12} yields the expression
\[u(4)=\frac{\Gamma(\frac{1}{3})^6}{2^{5/3}\pi^4}\,{}_7F_6\bigg({\frac{1}{6},\frac{1}{3},\frac{1}{3},\frac{1}{2},\frac{1}{2},\frac{1}{2},\frac{7}{6} \atop \frac{1}{6}, \frac{5}{6},\frac{5}{6},\frac{5}{6},1, \frac{7}{6}} \; \bigg| \;1\bigg),\]
which reduces to the ${}_5F_4$ generalized hypergeometric function given in (\ref{5f4}) due to the common parameters $1/6$ and $7/6$ (which is clear from the series representation (\ref{seriesrep}) of the generalized hypergeometric function).
\end{remark}

\begin{remark}
In the proof of Theorem \ref{thm1.1}, we arrived at several integral representations of $u(4)$. However, in terms of numerical computation, the representation (\ref{5f4}) of $u(4)$ in terms of a single ${}_5F_4$ function is particularly convenient since the generalized hypergeometric function is implemented in standard computational algebra packages. 
\end{remark}

It is natural to ask whether one can further simplify the expressions (\ref{4f3}) and (\ref{5f4}) for $u(4)$. The author did not find any such simplifications in the literature; however, this does not rule out the possibility of further simplifications. It is also natural to ask whether one can obtain a representation of $u(d)$ for $d\geq5$ in terms of generalized hypergeometric functions or other special functions that are simpler than the Lauricella function of type $C$. The following integral representations were obtained by  \cite{m56}: for $d\geq3$,
\begin{align}
u(d)&=\frac{d}{(2\pi)^d}\int_{(-\pi,\pi)^d} \bigg(d-\sum_{k=1}^d \cos x_k\bigg)^{-1} \, \mathrm{d}x_1\, \mathrm{d}x_2\cdots\, \mathrm{d}x_d \nonumber\\
&=\int_0^\infty \left[ I_0 \left( \frac {x}{d} \right) \right]^d  \mathrm{e}^{-x}  \, \mathrm{d}x. \nonumber
\end{align}
However, the author was unable to find evaluations of these integrals in the literature for $d\geq5$ (other than the known evaluation in terms of Lauricella function of type $C$ given by \cite{gnp24}).

We end with the following curiosity in which it is noted that in dimension 3 the return probability can also be expressed in terms of the generalized hypergeometric function with unit argument.

\begin{proposition}In dimension 3, we have that
\begin{align}
u(3)&=\frac{3\sqrt{2}}{\pi}\,{}_2F_1\bigg({\frac{1}{12},\frac{1}{2} \atop \frac{25}{24}} \; \bigg| \;1\bigg) \label{111} \\
&=\frac{1+\sqrt{3}}{2}\,{}_3F_2\bigg({\frac{1}{12},\frac{1}{2},\frac{11}{12} \atop 1,1} \; \bigg| \;1\bigg). \label{222}
\end{align}    
\end{proposition}

\begin{proof} We first derive formula (\ref{111}). By Gauss's multiplication formula \cite[Eq.\ 5.5.6]{olver}, we have that
\begin{align*}
\Gamma\Big(\frac{5}{24}\Big)\Gamma\Big(\frac{13}{24}\Big)\Gamma\Big(\frac{7}{8}\Big)=\frac{2\pi}{3^{1/8}}\Gamma\Big(\frac{5}{8}\Big), \quad
\Gamma\Big(\frac{7}{24}\Big)\Gamma\Big(\frac{5}{8}\Big)\Gamma\Big(\frac{23}{24}\Big)=\frac{2\pi}{3^{3/8}}\Gamma\Big(\frac{7}{8}\Big).   
\end{align*}
Applying these equations
to formula (\ref{for1}) gives that
\begin{align}\label{rel}
u(3)=  \frac{3\sqrt{2}}{\pi}\frac{\Gamma(\frac{11}{24})\Gamma(\frac{25}{24})}{\Gamma(\frac{13}{24})\Gamma(\frac{23}{24})} =\frac{\sqrt{2}}{8\pi}\frac{\Gamma(\frac{1}{24})\Gamma(\frac{11}{24})}{\Gamma(\frac{13}{24})\Gamma(\frac{23}{24})}, 
\end{align}
where in the second step we used the standard identity $\Gamma(x+1)=x\Gamma(x)$. Formula (\ref{111}) now follows by applying the identity
\begin{equation}\label{gauss}
2F_1(a,b;c;1)=\frac{\Gamma(c)\Gamma(c-a-b)}{\Gamma(c-a)\Gamma(c-b)}, \quad \mathrm{Re}(c-a-b)>0,    
\end{equation}
(see \cite[Eq.\ 15.4.20]{olver}) to equation (\ref{rel}).

We now prove formula (\ref{222}). By the reflection formula $\Gamma(x)\Gamma(1-x)=\pi/\sin(\pi x)$ (see \cite[Eq.\ 5.5.3]{olver}) we have that
\begin{align}
\Gamma\Big(\frac{1}{24}\Big)\Gamma\Big(\frac{11}{24}\Big)=\frac{\pi^2}{\sin(\frac{\pi}{24})\sin(\frac{11\pi}{24})\Gamma(\frac{13}{24})\Gamma(\frac{23}{24})}=\frac{8\pi^2}{(\sqrt{6}-\sqrt{2})\Gamma(\frac{13}{24})\Gamma(\frac{23}{24})},  \label{plug} 
\end{align}
where we used that $2\sin(\pi/24)\sin(11\pi/24)=\cos(5\pi/12)=(\sqrt{6}-\sqrt{2})/4$. Plugging (\ref{plug}) into (\ref{for2}) and simplifying gives that
\begin{align*}
u(3)&=\frac{1+\sqrt{3}}{2} \bigg[\frac{\sqrt{\pi}}{\Gamma(\frac{13}{24})\Gamma(\frac{23}{24})}\bigg]^2 =\frac{1+\sqrt{3}}{2} \bigg[{}_2F_1\bigg({\frac{1}{24},\frac{11}{24} \atop 1} \; \bigg| \;1\bigg)\bigg]^2=\frac{1+\sqrt{3}}{2}\,{}_3F_2\bigg({\frac{1}{12},\frac{1}{2},\frac{11}{12} \atop 1,1} \; \bigg| \;1\bigg),
\end{align*}
where we obtained the second equality using (\ref{gauss}) and the last equality was obtained by an application of Clausen's formula (see \cite[Eq.\ 16.12.2]{olver}). This completes the proof.
\end{proof}

\section*{Acknowledgements} The author is funded by EPSRC grant EP/Y008650/1. The author would like to thank the referee for their helpful comments and suggestions.

\end{document}